\documentclass[a4paper, english, 11pt]{amsart}

\usepackage[T1]{fontenc}
\usepackage{babel}
\usepackage{amsmath, amsfonts, amssymb, amsthm} 
\usepackage{url}      
\usepackage[numbers]{natbib}   
\usepackage{mathrsfs}

\numberwithin{equation}{section}

\newtheoremstyle{thm}{10pt}{5pt}{}{}{\bfseries}{}{ }{\thmname{#1}
  \thmnumber{#2} \thmnote{(#3)}}

\theoremstyle{thm}
\newtheorem{theorem}{Theorem}[section]

\newtheorem{lemma}[theorem]{Lemma}

\newtheorem{proposition}[theorem]{Proposition}

\newtheorem{remark}[theorem]{Remark}
\theoremstyle{remark}

\DeclareMathOperator{\dom}{dom}

\begin{document}

\title{The $K$-theory of some reduced inverse semigroup $C^\ast$-algebras.}
\author{Magnus Dahler Norling}
\address{Magnus Dahler Norling, Institute of Mathematics, University of Oslo, P.b. 1053 Blindern, 0316 Oslo, Norway.}
\email{magnudn@math.uio.no}

\begin{abstract}
We use a recent result by Cuntz, Echterhoff and Li about the $K$-theory of certain reduced $C^\ast$-crossed products to describe the $K$-theory of $C^\ast_r(S)$ when $S$ is an inverse semigroup satisfying certain requirements. A result of Milan and Steinberg allows us to show that $C^\ast_r(S)$ is Morita equivalent to a crossed product of the type handled by Cuntz, Echterhoff and Li. We apply our result to graph inverse semigroups and the inverse semigroups of one-dimensional tilings.
\end{abstract}

\maketitle

\section{Introduction}
A \emph{semigroup} is a set $P$ with an associative binary operation. It is a \emph{monoid} if it has an identity element. An \emph{inverse semigroup} is a semigroup $S$ where for every $s\in S$ there is a unique element $s^\ast\in S$ satisfying
\[
ss^\ast s=s\qquad s^\ast ss^\ast=s^\ast.
\]
We will assume that every inverse semigroup we are working with has a zero element $0=0_S$ and that $S$ is countable. Let $E=E(S)=\{ss^\ast:s\in S\}$ be the set of idempotents in $S$. Then $E$ is a commutative idempotent semigroup, i.e. a \emph{semilattice}. Let $S^\times=S\setminus\{0\}$, and if $S$ does not have an identity element let $S^1=S\cup\{1\}$ with the obvious operations. If $S$ has an identity element, let $S^1=S$. See \cite{lawson98, paterson99} for general references on inverse semigroups.

Every inverse semigroup comes equipped with a \emph{natural partial order} given by $s\leq t$ if there is some $e\in E$ such that $s=et$, equivalently, there is some $f\in E$ such that $s=tf$. This is also equivalent to $s=ss^\ast t$ and to $s=ts^\ast s$.The inverse semigroup $S$ is said to be $0$-$F$-inverse if every element in $S^\times$ is beneath a unique maximal element with respect to the natural partial order. In this case, let $M(S)$ denote the set of maximal elements in $S$.

Let $G$ be a group. A \emph{morphism} \cite{kellendonk_lawson04} from $S^\times$ to $G$ (sometimes called \emph{grading} of $S$) is a map $\sigma:S^\times\to G$ satisfying
\[
 \sigma(st)=\sigma(s)\sigma(t)\mbox{ whenever }st\neq 0
\]
Note that $\sigma(e)=1$ for every $e\in E$. The morphism $\sigma$ is said to be \emph{idempotent pure} if $\sigma(s)=1$ implies $s\in E$ for any $s\in S$. It is easy to see that if $S$ is $0$-$F$-inverse and $\sigma:S^\times\to G$ is a morphism that is injective on $M(S)$, then $\sigma$ is idempotent pure.

Associated to every inverse semigroup $S$ there is a $C^\ast$-algebra $C^\ast(S)$ that is universal for all $\ast$-representations of $S$. Similarly there is a reduced $C^\ast$-algebra $C^\ast_r(S)$ which is the image of the left regular representation
\[
\Lambda:C^\ast(S)\to B(\ell^2(S)).
\]
We use here the convention that all $\ast$-representations of $S$ are to send the $0$-element of $S$ to $0$. This convention coincides with the one used in \cite{milan_steinberg11}. See otherwise \cite{paterson99} for an introduction to $C^\ast$-algebras associated with inverse semigroups.

The main purpose of the paper is to provide a formula describing the $K$-theory of $C^\ast_r(S)$. The key ingredient will be a recent result from \cite{cuntz_echterhoff_li12} concerning the $K$-theory of some crossed products $C_0(\Omega)\rtimes_r G$ where $G$ is a group and $\Omega$ is a locally compact totally disconnected (i.e. zero-dimensional) Hausdorff space. We will prove the following main theorem by showing that $C^\ast_r(S)$ is strongly Morita equivalent to such a crossed product under suitable conditions.

\begin{theorem}\label{thm:maintheorem}
Let $S$ be a $0$-$F$-inverse semigroup, and suppose there exists a group $G$ and a morphism $\sigma:(S^1)^\times\to G$ that is injective on the set of maximal elements in $S^1$. Let $\approx$ be the equivalence relation on $E$ given by $e\approx f$ if there is some $s\in S$ such that $s^\ast s=e$ and $ss^\ast=f$. Let $\widetilde{E}$ be the set $E^\times/\approx$. For each $e\in E$, let $G_e=\{\sigma(s):ss^\ast=s^\ast s=e\}$. If $G$ is a-T-menable \cite{aTmenability01}, then
\[
 K_\ast(C^\ast_r(S))\simeq \bigoplus_{[e]\in\widetilde{E}}K_\ast(C^\ast_r(G_e)).
\]
\end{theorem}

Many important inverse semigroups are $0$-$F$ inverse. One example is the left inverse hull associated to a submonoid og a group satisfying the Toeplitz condition of \cite{li12}. Using this we will show that our main theorem generalizes \cite[Corollary 4.9]{cuntz_echterhoff_li12} which describes the $K$-theory of the reduced $C^\ast$-algebra of such monoids. Other examples are graph inverse semigroups \cite{paterson02} and some tiling and point set inverse semigroups \cite{kellendonk_lawson04}. The reduced $C^\ast$-algebra of the graph inverse semigroup is the Toeplitz $C^\ast$-algebra of the graph as defined in \cite{raeburn05}. We calculate the $K$-theory for some of these examples. In general, the reduced $C^\ast$-algebra of $S$ is not the only one of interest, but also its tight $C^\ast$-algebra as defined in \cite{exel08}. For instance, the Cuntz-Krieger graph $C^\ast$-algebras and some of the tiling $C^\ast$-algebras \cite{kellendonk_putnam00} are the tight $C^\ast$-algebras of their respective inverse semigroups. The tight $C^\ast$-algebra is a quotient of the reduced $C^\ast$-algebra in the stated examples. We hope that the results of this paper may be useful in the future when trying to describe the $K$-theory of these tight $C^\ast$-algebras.

\section{Preliminaries on the construction of the Morita enveloping action}

In this and the following section we will fix a $0$-$F$-inverse semigroup $S$, and we will suppose that there exists a group $G$ and a morphism $\sigma:(S^1)^\times\to G$ that is injective on $M(S^1)$. Note that if $S$ is any inverse semigroup and $\sigma:S^\times\to G$ is a morphism such that $\sigma^{-1}(g)$ has a unique maximal element for each $g\in G$, then $S$ is $0$-$F$-inverse.

Our exposition in the first part of this section follows \cite{paterson99,exel08}. Let $S$ be an inverse semigroup and let $E$ be its semilattice of idempotents. A \emph{character} on $E$ is a homomorphism $\phi:E\to \{0,1\}$ that satisfies $\phi(0)=0$ and $\phi(e)=1$ for at least one $e\in E$. Here $\{0,1\}$ is considered a semilattice in the obvious way. Let $\widehat{E}\subset \{0,1\}^E$ be the set of all characters on $E$ with the relative topology inherited from the product topology on $\{0,1\}^E$. The locally compact space $\widehat{E}$ is called the \emph{spectrum} of $E$. For each $e\in E$, let
\[
 D_e=\{\phi\in\widehat{E}:\phi(e)=1\}.
\]
Then $D_e$ is a compact open subset of $\widehat{E}$. For each $e\in E$, let $\phi_e$ be the character given by
\begin{equation}\label{eq:minimalcharacter}
 \phi_e(f)=\begin{cases}
            1\mbox{ if }e\leq f\\
            0\mbox{ otherwise.}
           \end{cases}
\end{equation}
For any $e,f\in E^\times$ it is easy to show that $\phi_e=\phi_f$ if and only if $e=f$. Recall  that there is an action $\theta$ of $S$ by partial homeomorphisms on $\widehat{E}$ given by
\begin{align*}
\theta_s:D_{s^\ast s}&\to D_{ss^\ast},\\
\theta_s(\phi)(e)&=\phi(s^\ast es),\quad\phi\in D_{s^\ast s},e\in E.
\end{align*}
One can show that $\theta_s$ is a homeomorphism for each $s\in S$, and $\theta_s^{-1}=\theta_{s^\ast}$. In \cite{milan_steinberg11} the partial action $\beta$ of $G$ on $\widehat{E}$ is defined as follows
\[
\beta_g=\bigcup_{\sigma(s)=g}\theta_s.
\]
Since we are working with a $0$-$F$-inverse semigroup and since $\sigma$ is injective on $M(S^1)$, this expression can be simplified. Let $s_g$ be the uniqe element in $M(S^1)$ satisfying $\sigma(s_g)=g$. Then
\[
\beta_g=\theta_{s_g}.
\]
We now recall from \cite{abadie03} the construction of the Morita enveloping action $(\Omega, G, \tau)$ of a partial action $(X,G,\beta)$. Define an equivalence relation $\sim$ on $G\times X$ by $(g,x)\sim (h,y)$ if $x\in \dom(\beta_{h^{-1}g})$ and $\beta_{h^{-1}g}(x)=y$. Define $\Omega=\Omega_X=(G\times X)/\sim$ with the quotient topology, and let $[g,x]$ denote the equivalence class of $(g,x)\in G\times X$. Let $G$ act on $\Omega$ by
\[
\tau_g([h,x])=[gh,x].
\]
By \cite[Theorem 1.1]{abadie03}, the map $\iota:x\mapsto [1,x]$ defines a homeomorphism $X\to\iota(X)$, with $\iota(X)$ open in $\Omega$. Moreover, for all $g\in G$ and $x\in d(\theta_g)$,
\begin{equation}\label{eq:intertwine}
\iota(\theta_g(x))=\tau_g(\iota(x))
\end{equation}
It also follows that the orbit of $\iota(X)$ in $\Omega$ is all of $\Omega$. We will from here on omit writing $\iota$, and view $X$ as a subset of $\Omega$.

Note in that \cite{milan_steinberg11} Milan and Steinberg assume that $G$ is the universal group of $S$, but we want a more general result and do not want to demand that $G$ is the universal group. Because of this we have to redo their proof that $C^\ast_r(S)$ is strongly Morita equivalent to $C_0(\Omega)\rtimes_r G$ using a result from \cite{khoshkam_skandalis02}.

The \emph{universal groupoid} $\mathcal{G}_u$ of $S$ is the groupoid of germs for $\theta$ \cite{exel08}, and can be defined as follows. For $s,t\in S^\times$ and $\phi\in D_{s^\ast s}, \psi\in D_{t^\ast t}$, let $(s,\phi)\sim (t,\psi)$ if and only if $\phi=\psi$ and there exists $u\leq s,t$ such that $\phi\in D_{u^\ast u}$. Let $[s,\phi]$ denote the equivalence class of $(s,\phi)$. Then define
\begin{align*}
\mathcal{G}_u&=\{[s,\phi]:s\in S^\times,\phi\in D_{s^\ast s}\},\\
\mathcal{G}_u^{(2)}&=\{([s,\phi],[t,\psi])\in\mathcal{G}_u\times\mathcal{G}_u:\phi=\theta_t(\psi)\}
\end{align*}
For $([s,\phi],[t,\psi])\in\mathcal{G}_u^{(2)}$, let
\begin{align*}
[s,\phi][t,\psi]&=[st,\psi],\\
[s,\phi]^{-1}&=[s^\ast,\theta_s(\phi)]
\end{align*}
By \cite[Proposition 4.11]{exel08}, the unit space $\mathcal{G}_u^{(0)}$ can be identified with $\widehat{E}$ by sending $[e,\phi]$ to $\phi$, where $e\in E$ and $\phi\in D_e$. The $d$ and $r$ maps from $\mathcal{G}_u$ to $\mathcal{G}_u^{(0)}=\widehat{E}$ are given by
\begin{align*}
d([s,\phi])&=\phi,\\
r([s,\phi])&=\theta_s(\phi).
\end{align*}
For each $s\in S$ and open $U\subset D_{s^\ast s}$, let
\[
 \Theta(s,U)=\{[s,\phi]:\phi\in U\}
\]
The collection
\[
\{\Theta(s,U):s\in S, U\subset D_{s^\ast s}\mbox{ is open}\}
\]
is a basis for the topology on $\mathcal{G}_u$. Then for each $s\in S$, the restrictions
\[
d:\Theta(s,D_{s^\ast s})\to D_{s^\ast s},\qquad r:\Theta(s,D_{s^\ast s})\to D_{ss^\ast}
\]
are homeomorphisms. The reduced $C^\ast$-algebra of $S$ is isomorphic to the reduced $C^\ast$-algebra of $\mathcal{G}_u$ \cite{paterson99}.

In \cite{khoshkam_skandalis02}, a \emph{continuous cocycle} on a groupoid $\mathcal{G}$ is defined to be a continuous groupoid homomorphism $\rho:\mathcal{G}\to H$, where $H$ is a group. The cocycle $\rho$ is said to be
\begin{enumerate}
\item \emph{faithful} if the map $\mathcal{G}\mapsto\mathcal{G}^{(0)}\times H\times\mathcal{G}^{(0)}$ given by $\gamma\mapsto (r(\gamma),\rho(\gamma),d(\gamma))$ is injective.
\item \emph{closed} if the map $\gamma\mapsto (r(\gamma),\rho(\gamma),d(\gamma))$ is closed.
\item \emph{transverse} if the map $H\times\mathcal{G}\to H\times\mathcal{G}$ given by $(\gamma,g)\mapsto (g\rho(\gamma),d(\gamma))$ is open.
\end{enumerate}

If $\rho$ is faithful, closed and transverse, one defines the \emph{enveloping action} of $\rho$ as follows. Let $\sim$ be the equivalence relation on $H\times\mathcal{G}^{(0)}$ given by $(g,x)\sim (h,y)$ if there exists a $\gamma\in\mathcal{G}$ with $d(\gamma)=y$, $r(\gamma)=x$ and $\rho(\gamma)=h^{-1}g$. Let $\Omega$ be the space $(H\times\mathcal{G}^{(0)})/\sim$ with the quotient topology, and let $\tau$ be the action of $H$ on $\Omega$ given by $\tau_g([h,x])=[gh,x]$. By \cite[Lemma 1.7 and Theorem 1.8]{khoshkam_skandalis02}, $\mathcal{G}$ is Morita equivalent to the transformation groupoid $H\ltimes\Omega$, and so $C^\ast_r(\mathcal{G})$ is strongly Morita equivalent to $C_0(\Omega)\rtimes_{\tau,r}H$.

Let $\mathcal{G}_u$ be the universal groupoid of $S$ and let $\sigma:(S^1)^\times\to G$ be the morphism specified earlier. Define a map $\rho:\mathcal{G}_u\to G$ by
\[
\rho([s,\phi])=\sigma(s).
\]
To see that this is well defined, note that if $0\neq u\leq s,t$, then $\sigma(u)=\sigma(s)=\sigma(t)$. It is straightforward to check that $\rho$ is a continuous cocycle. Inspection reveals that the enveloping action of $\rho$ is the Morita enveloping action of $\beta$. To see this, note that the equivalence relation one obtains on $G\times\widehat{E}$ is the same in either case. It remains to prove the following proposition. The proof closely follows the proof of \cite[Propositions 3.6 and 3.9]{khoshkam_skandalis02}.

\begin{proposition}
The cocycle $\rho:\mathcal{G}_u\to G$ is faithful, closed and transverse.
\end{proposition}
\begin{proof}
We prove faithfulness first. Let $\gamma=[s,\phi]\in\mathcal{G}_u$. Then
\[
(r(\gamma)\rho(\gamma),d(\gamma))=(\theta_s(\phi),\sigma(s),\phi).
\]
Suppose $(\theta_s(\phi),\sigma(s),\phi)=(\theta_t(\psi),\sigma(t),\psi)$. Then $\phi=\psi$ and $\sigma(s)=\sigma(t)$. Since $\phi\in D_{t^\ast t}\cap D_{s^\ast s}=D_{t^\ast ts^\ast s}$, $t^\ast ts^\ast s\neq 0$. So $ts^\ast\neq 0$, and 
\[
\sigma(ts^\ast)=\sigma(t)\sigma(s)^{-1}=\sigma(t)\sigma(t)^{-1}=1.
\]
Thus $ts^\ast$ is idempotent. Let $u=ts^\ast s$. Then $0\neq u\leq s,t$ and $u^\ast u=s^\ast st^\ast t$, so $\phi\in D_{u^\ast u}$. It follows that $[s,\phi]=[t,\psi]$.

Next we prove closedness. Let $X\subset\mathcal{G}_u$ be closed. Then
\[
 (\rho,r,d)(X)=\bigcup_{g\in\sigma(S^\times)}\{g\}\times\{r(\gamma),d(\gamma)):\gamma\in X,\rho(\gamma)=g\}.
\]
Since $G$ is discrete, we only have to show that
\[
\{(d(\gamma),r(\gamma)):\gamma\in X,\rho(\gamma)=g\}
\]
is closed for each $g\in\sigma(S^\times)$. If $g=1$, then this set is just $(X\cap\widehat{E})^2$. Since $\widehat{E}$ is a closed subset of $\mathcal{G}_u$, this is closed. If $g\neq 1$, then $\sigma^{-1}(g)$ has a maximal element $s_g$, and $[s,\phi]=[s_g,\phi]$ for each $[s,\phi]\in\mathcal{G}_u$ with $\sigma(s)=g$, so
\[
\{(d(\gamma),r(\gamma)):\gamma\in X,\rho(\gamma)=g\}=\{(d(\gamma),r(\gamma)):\gamma\in X\cap\Theta(s,D_{s_g^\ast s_g})\}
\]
Since $d$ and $r$ are homeomorphisms of $\Theta(s_g, D_{s_g^\ast s_g})$ onto closed subsets of $\widehat{E}$, the assertion follows.

Finally, we prove transversity. Since $G$ is discrete it is sufficient to show that the set
\[
\{d(\gamma):\gamma\in\mathcal{G}_u,\rho(\gamma)=g\}=\bigcup_{s\in\sigma^{-1}(g)}D_{s^\ast s}
\]
is open for each $g\in\sigma(S^\times)$, but this is obviously true.
\end{proof}

\begin{remark}
It is shown in \cite{milan_steinberg11} that Morita equivalence with the enveloping action holds for a slightly larger class of inverse semigroups than the one we have considered here, but this restriction is needed in some of the proofs of the next section.
\end{remark}

\section{Proof of the main theorem}
Recall that a locally compact Hausdorff space is totally disconnected if and only if it has a basis of compact open subsets. For a totally disconnected space $X$, let $\mathcal{U}_c(X)$ be the family (actually a Boolean algebra) of compact open subsets of $X$. In \cite{cuntz_echterhoff_li12}, a subfamily $\mathcal{V}\subset\mathcal{U}_c(X)$ is said to be a \emph{generating family} for $\mathcal{U}_c(X)$ if $\mathcal{U}_c(X)$ is the smallest family of compact open sets that contains $\mathcal{V}$ and is closed under finite intersections, finite unions and taking difference sets (i.e. $\mathcal{V}$ generates $\mathcal{U}_c(X)$ as a Boolean algebra). The family $\mathcal{V}$ is said to be \emph{independent} if for any $U,U_1,\ldots,U_n\in\mathcal{V}$, $U=\bigcup_{i=1}^n U_n$ implies that $U=U_j$ for some $1\leq j\leq n$. If $\mathcal{V}$ is an independent generating family for $\mathcal{U}_c(X)$ and $\mathcal{V}\cup\{\varnothing\}$ is closed under finite intersections, then $\mathcal{V}$ is said to be a \emph{regular basis} for $X$. For notational convenience we will require that $\varnothing\notin\mathcal{V}$. The following theorem is a part of the statement of \cite[Corollary 3.14]{cuntz_echterhoff_li12}.

\begin{theorem}[\cite{cuntz_echterhoff_li12}]\label{thm:CELktheory}
Let $X$ be a separable totally disconnected locally compact Hausdorff space with an action $\alpha$ of a group $G$ and suppose that $X$ has a $G$-invariant regular basis $\mathcal{V}$. Suppose also that $G$ satisfies the Baum-Connes conjecture with coefficients in $C_0(X)$ and $C_0(\mathcal{V})$. Then
\[
K_\ast(C_0(X)\rtimes_{\alpha,r} G)\simeq\bigoplus_{[v]\in\mathcal{V}\setminus G}K_\ast(C^\ast_r(G_v))
\]
where $\mathcal{V}\setminus G$ is the set of orbit classes in $\mathcal{V}$ and $G_v$ is the stabilizer group of $v\in\mathcal{V}$.
\end{theorem}

We will use Theorem \ref{thm:CELktheory} to prove Theorem \ref{thm:maintheorem}. The way we will do it is to apply Theorem \ref{thm:CELktheory} to the enveloping action $(\Omega,G,\tau)$ of $(\widehat{E},G,\beta)$. We will show that
\begin{equation}\label{eq:regularbasis}
\mathcal{V}=\{\tau_g(D_e):e\in E^\times,g\in G\}
\end{equation}
is a $G$-invariant regular basis for $\Omega$. It is shown in \cite{higson_kasparov01} that if $G$ is a-T-menable (i.e. it satisfies the Haagerup property), then it satisfies the strong Baum-Connes conjecture with arbitrary separable coefficients. The statement of Theorem \ref{thm:maintheorem} could have been made more general by not requiring that $G$ is a-T-menable, but only that it satisfies the Baum-Connes conjecture with coefficients in $C_0(\Omega)$ and $C_0(\mathcal{V})$. However, that would have made the statement more complicated by making explicit reference to the enveloping action instead of just utilizing it in the proof. The next proposition will allow us to translate statements about the partial action into statements about the enveloping action.

\begin{proposition}\label{prop:basistranslate}
Let $X\subset Y$ be locally compact Hausdorff spaces such that $X$ is totally disconnected and open in $Y$. Let $\alpha$ be an action of a discrete group $G$ on $Y$ such that $\bigcup_{g\in G}\alpha_g(X)=Y$. Let $\mathcal{J}$ be a collection of compact open subsets of $X$. Suppose $\alpha_g(p)\cap X\in\mathcal{J}\cup\{\varnothing\}$ for every $p\in\mathcal{J}$ and $g\in G$. Let $\mathcal{W}=\{\alpha_g(p):g\in G,p\in\mathcal{J}\}$. Then
\begin{enumerate}
 \item If $\mathcal{J}\cup\{\varnothing\}$ is closed under finite intersections, so is $\mathcal{W}\cup\{\varnothing\}$.
 \item If $\mathcal{J}$ is independent, so is $\mathcal{W}$.
 \item If $\mathcal{J}$ is a generating family for $\mathcal{U}_c(X)$, then $\mathcal{W}$ is a generating family for $\mathcal{U}_c(Y)$ and $Y$ is totally disconnected.
\end{enumerate}
So if $\mathcal{J}$ is a regular basis for $X$, then $\mathcal{W}$ is a $G$-invariant regular basis for $Y$.
\end{proposition}
\begin{proof}
(i) It is sufficient to show that the intersection of two elements in $\mathcal{W}$ is in $\mathcal{W}\cup\{\varnothing\}$. Let $p,q\in\mathcal{J}\cup\{\varnothing\}$ and $g,h\in G$. Let $r=\alpha_{g^{-1}h}(q)\cap X\in\mathcal{J}\cup\{\varnothing\}$. Then
\begin{align*}
\alpha_g(p)\cap\alpha_h(q)&=\alpha_g(p\cap\alpha_{g^{-1}h}(q))=\alpha_g(p\cap(\alpha_{g^{-1}h}(q)\cap X))\\&=\alpha_g(p\cap r).
\end{align*}
Since $\mathcal{J}\cup\{\varnothing\}$ is closed under finite intersections, $p\cap r\in\mathcal{J}\cup\{\varnothing\}$, so $\alpha_g(p\cap r)\in\mathcal{W}\cup\{\varnothing\}$.

(ii) Let $p,p_1,\ldots,p_n\in\mathcal{J}$ and let $g,g_1,\ldots,g_n\in G$. Suppose
\[
\bigcup_{i=1}^n\alpha_{g_i}(p_i)=\alpha_g(p).
\]
Then
\[
\bigcup_{i=1}^n\alpha_{g^{-1}g_i}(p_i)=p.
\]
For each $i$, $\alpha_{g^{-1}g_i}(p_i)\subset p\subset X$, so $\alpha_{g^{-1}g_i}(p_i)\cap X=\alpha_{g^{-1}g_i}(p_i)\in\mathcal{J}$. It follows by the independence of $\mathcal{J}$ that $p=\alpha_{g^{-1}g_j}(p_j)$ for at least one $1\leq j\leq n$, so $\alpha_{g_j}(p_j)=\alpha_g(p)$.

(iii) First, any element in $\mathcal{W}$ is compact open in $Y$ since $X$ is open in $Y$. Let $U\subset Y$ be compact open. We have
\[
 U=U\cap\bigcup_{g\in G}\alpha_g(X)=\bigcup_{g\in G}\alpha_g(\alpha_{g^{-1}}(U)\cap X).
\]
and $\alpha_{g^{-1}}(U)\cap X$ is open in $Y$ for each $g\in G$. By the compactness of $U$, there is some finite subset $F\subset G$ such that
\[
U=\bigcup_{g\in F}\alpha_g(\alpha_{g^{-1}}(U)\cap X).
\]
For each $g\in G$, $\alpha_{g^{-1}}(U)\cap X$ is compact open in $X$. So since $\mathcal{J}$ is a generating family for $\mathcal{U}_c(X)$, it follows that $\mathcal{W}$ is a generating family for $\mathcal{U}_c(Y)$. As $X$ is totally disconnected, any open set in $X$ is a union of elements in $\mathcal{U}_c(X)$. It follows by a similar argument that any open set in $Y$ is a union of elements in $\mathcal{U}_c(Y)$, so $Y$ is totally disconnected.
\end{proof}

\begin{lemma}\label{lem:gaction}
For any $e\in E$ and $g\in \sigma(S^\times)$ we have $\tau_g(D_e)\cap\widehat{E}=D_{s_ges_g^\ast}$. Moreover $\tau_g(D_e)\subset\widehat{E}$ if and only if $e\leq s_g^\ast s_g$.
\end{lemma}
\begin{proof}
For any $g\in \sigma(S^\times)$ and $\phi\in\widehat{E}$, we have that $[g,\phi]\in\widehat{E}$ if and only if $[g,\phi]=[1,\psi]$ for some $\psi\in\widehat{E}$ if and only if $\phi\in D_{s_g^\ast s_g}$ and $\theta_{s_g}(\phi)=\psi$. Then $\tau_g(D_e)\cap\widehat{E}=\theta_{s_g}(D_e\cap D_{s_g^\ast s_g})$. By \cite[Equation (10.3.1)]{exel08} this set is equal to $D_{s_ges_g^\ast}$. If $e\leq s_g^\ast s_g$, then $D_e\subset D_{s_g^\ast s_g}$, so $\tau_g(D_e)=\theta_{s_g}(D_e)\subset\widehat{E}$. Suppose $\tau_g(D_e)\subset\widehat{E}$. Then $\phi\in D_{s_g^\ast s_g}$ for each $\phi\in D_e$, so $D_e\subset D_{s_g^\ast s_g}$ and $e\leq s_g^\ast s_g$.
\end{proof}

Given any $e\in E^\times$ and finite subset $Z\subset eE=\{f\in E:f\leq e\}$, let
\[
 D_{(e,Z)}=D_e\setminus\left(\bigcup_{z\in Z}D_z\right).
\]
Note that since $D_x$ is compact open for each $x\in E$, so is $D_{(e,Z)}$. Recall from \cite[Section 4.3]{paterson99} that the family
\[
\mathcal{T}=\{D_{(e,Z)}:e\in E^\times, Z\subset eE\mbox{ is a finite subset}\}
\]
is a basis for the topology on $\widehat{E}$. We now see that any compact open subset of $\widehat{E}$ can be written as a finite union of elements in $\mathcal{T}$: Indeed, since $\mathcal{T}$ is a basis for the topology on $\widehat{E}$, any open set $U$ in in $\widehat{E}$ can be written as a union of elements in $\mathcal{T}$. These elements form a cover for $U$, so if $U$ is compact, one can rewrite $U$ as a finite union of such elements.

\begin{proposition}\label{prop:visbasis}
The family $\mathcal{V}$ defined in equation \eqref{eq:regularbasis} is a $G$-invariant regular basis for $\Omega$.
\end{proposition}
\begin{proof}
If we can show that the family $\mathcal{J}=\{D_e:e\in E^\times\}$ is a regular basis for $\widehat{E}$, the result will follow from Proposition \ref{prop:basistranslate} and Lemma \ref{lem:gaction}.

We know that $\mathcal{J}\cup\{\varnothing\}$ is closed under finite intersections, since for any $e,f\in E$, $D_e\cap D_f=D_{ef}$ (and $D_0=\varnothing$). To show that $\mathcal{J}$ is independent, suppose that $e,e_1,\ldots,e_n\in E^\times$ satisfy $\bigcup_{i=1}^n D_{e_i}=D_e$. For each $i$, $D_{e_i}\subset D_e$, so $e_i\leq e$. Moreover, $\phi_e\in D_{e_j}$ for at least one $1\leq j\leq n$. Then $\phi_e(e_j)=1$, so $e\leq e_j$ by the definition of $\phi_e$. Thus $e=e_j$ and $D_e=D_{e_j}$. Finally, we need that $\mathcal{J}$ is a generating set for $\mathcal{U}_c(\widehat{E})$. This follows from the fact that any compact open subset of $\widehat{E}$ is a finite union of elements in $\mathcal{T}$.
\end{proof}

Recall from Theorem \ref{thm:maintheorem} that we defined a relation $\approx$ on $E^\times$ by $e\approx f$ if there is some $s\in S^\times$ such that $s^\ast s=e$ and $ss^\ast=f$.

\begin{lemma}\label{lem:orbitclass}
Let $e,f\in E^\times$. Then $D_e$ and $D_f$ are in the same $G$-orbit in $\mathcal{V}$ if and only if $e\approx f$.
\end{lemma}
\begin{proof}
Suppose $g\in G$ satisfies $\tau_g(D_e)=D_f$. Then by Lemma \ref{lem:gaction} $e\leq s_g^\ast s_g$ and $s_ges_g^\ast=f$. Let $s=s_ge$. Then $s^\ast s=es_g^\ast s_g e=e$ since $e\leq s_g^\ast s_g$, and $ss^\ast = s_g ees_g^\ast=s_ges_g^\ast =f$.

Suppose to the contrary that there is some $s\in S$ such that $s^\ast s=e$ and $ss^\ast=f$. Let $g=\sigma(s)$. Then $s\leq s_g$. It follows that $s_ge=s_gs^\ast s=s$, so $e\leq s_g^\ast s_g$, and $f=ss^\ast=s_ges_g^\ast$. So $\tau_g(D_e)=D_f$.
\end{proof}

Given $e\in E^\times$, let
\begin{align*}
G_e&=\{g\in\sigma(S^\times):e\leq s_g^\ast s_g,\, s_ges_g^\ast=e\}\\
&=\{\sigma(s):s\in M(S^1),e\leq s^\ast s, ses^\ast=e\}\\
&=\{\sigma(s):s\in S^\times, e\leq s^\ast s, ses^\ast=e\}\\
&=\{\sigma(s):s\in S^\times, ss^\ast=s^\ast s=e\}.
\end{align*}
We leave to the reader to prove that these definitions are equivalent.

\begin{lemma}\label{lem:stabgroup}
Let $e\in E^\times$. Then $G_e$ is the stabilizer group of $D_e\in\mathcal{V}$.
\end{lemma}
\begin{proof}
Assume $\tau_g(D_e)=D_e$. Then $\tau_g(D_e)\subset\widehat{E}$, so by Lemma \ref{lem:gaction} $e\leq s_g^\ast s_g$, $D_e=D_{s_ges_g^\ast}$ and $e=s_ges_g^\ast$. The converse implication follows from the same lemma.
\end{proof}

We can now complete the proof of the main theorem. Since $G$ is a-T-menable, Theorem \ref{thm:CELktheory} together with Proposition \ref{prop:visbasis} gives us that
\[
 K_\ast(C^\ast_r(S))\simeq K_\ast(C_0(\Omega)\rtimes_r G)\simeq \bigoplus_{[v]\in\mathcal{V}\setminus G}K_\ast(C^\ast_r(G_v))
\]
For any $v\in\mathcal{V}$ we have that $v=\tau_g(D_e)$ for some $e\in E^\times$, so $[v]=[D_e]$. Thus it is sufficient to take a direct sum over $\{[D_e]:e\in E^\times\}$, which can be identified with $\widetilde{E}$ by Lemma \ref{lem:orbitclass}. The rest follows from Lemma \ref{lem:stabgroup}.

\section{Connections to submonoids of groups}
Given a set $X$, let $I(X)$ be the inverse monoid of all partial bijections on $X$. Here a partial bijection is a bijection $f:d(f)\to r(f)$ with $d(f),r(f)\subset X$, $f^\ast =f^{-1}$ and the product is given by composition wherever it makes sense. This product may result in the empty function, which is the $0$-element of $I(X)$.

We will now explain how the ``$0$-$F$-inverse'' condition for inverse semigroups is related to the Toeplitz condition for subsemigroups of groups as defined in \cite{li12,cuntz_echterhoff_li12}. Let $P$ be a submonoid of the (discrete) group $G$. For each $p\in P$, let $\lambda_p:P\to pP$ be the map given by $\lambda_p(q)=pq$. Then $\lambda_p\in I_P$ since $P$ is left cancellative. Let $I_P^\lambda$ be the inverse subsemigroup of $I_P$ generated by $\{\lambda_p:p\in P\}\cup\{0\}$. Since $P$ is a monoid,  $I_P^\lambda$ is also a monoid. $I_P^\lambda$ is called the \emph{left inverse hull} of $P$ \cite{meakin05}. Each idempotent $e\in E(I_P^\lambda)$ is the identity function on some subset $X\subset P$. Let $\mathcal{J}=\{d(e):e\in E(I_P^\lambda)\}$. Then it is easy to show that $E(I_P^\lambda)$ and $(\mathcal{J},\cap)$ are isomorphic as semilattices.

Let $\{\varepsilon_p:p\in P\}$ be the canonical basis for $\ell^2(P)$. For each $p\in P$, let $V_p\in B(\ell^2(P))$ be given by $V_p\varepsilon_q=\varepsilon_{pq}$. Define $C^\ast_r(P)$ to be the $C^\ast$-algebra generated by $\{V_p:p\in P\}$. It was shown in \cite[Theorem 3.2.14]{norling12} that if $\mathcal{J}$ is independent, there is an isomorphism $C^\ast_r(I_P^\lambda)\to C^\ast_r(P)$ given by $\Lambda(\lambda_p)\mapsto V_p$.

As shown in \cite[Proposition 3.2.11]{norling12}, one can define an idempotent pure morphism $\sigma:(I_P^\lambda)^\times\to G$ by $\sigma(\lambda_p)=p$. Moreover, any group $H$ and idempotent pure morphism $\sigma:I_P^\lambda\to H$ gives rise to an embedding $P\to H$ given by $p\mapsto\sigma(\lambda_p)$. Note that for any $f\in I_P^\lambda$ and $p\in d(f)$,
\begin{equation}\label{eq:gradingaction}
f(p)=\sigma(f)p
\end{equation}
Let $\{\varepsilon_g:g\in G\}$ be the canonical basis for $\ell^2(G)$. Let $U_g\in B(\ell^2(G))$ be given by $U_g\varepsilon_h=\varepsilon_{gh}$ (i.e. $g\mapsto U_g$ is the left regular representation of $G$). Let $E_P$ be the orthogonal projection of $\ell^2(G)$ onto the subspace $\ell^2(P)$. By \cite[Definition 4.1]{li12}, $P\subset G$ satisfies the \emph{Toeplitz condition} if for any $g\in G$ with $E_P U_g E_P\neq 0$ there exist $p_1,\ldots,p_n,q_1,\ldots,q_n\in P$ such that $E_P U_g E_P=V_{q_1}^\ast V_{p_1}\cdots V_{q_n}^\ast V_{p_n}$.

\begin{proposition}\label{prop:toeplitz}
Let $P$ be a submonoid of the group $G$. Then $P\subset G$ satisfies the Toeplitz condition if and only if $I_P^\lambda$ is $0$-$F$-inverse and $\sigma$ is injective on the set of maximal elements in $I_P^\lambda$.
\end{proposition}
\begin{proof}
For $g\in G$, let $\alpha_g:(g^{-1})P\cap P\to gP\cap P$ be given by $\alpha_g(p)=gp$. Then $\alpha_g\in I_P$. We claim that $P\subset G$ satisfies the Toeplitz condition if and only if $\alpha_g\in I_P^\lambda$ for each $g\in G$. To verify this, let $\omega:I_P\to B(\ell^2(P))$ be given by
\[
\omega(f)\varepsilon_p=\begin{cases}
\varepsilon_{f(p)} \mbox{ if }p\in d(f)\\
0\mbox{ otherwise.}
\end{cases}
\]
Note that $\omega$ is injective and that $\omega(\lambda_p)=V_p$ for each $p\in P$. Moreover, for any $g\in G$, it is easy to verify that $\omega(\alpha_g)=E_P U_g E_P$.

Fix $g\in G$ and assume without loss of generality that $\alpha_g\neq 0$. Suppose $\alpha_g\in I_P^\lambda$. Clearly $\sigma(\alpha_g)=g$. Suppose $f\in I_P^\lambda$ satisfies $\sigma(f)=g$. By equation \eqref{eq:gradingaction} it follows that $f\leq\alpha_g$. This implies that $\alpha_g$ is the unique maximal element of $\sigma^{-1}(g)$.

Conversely, suppose $I_P^\lambda$ is $0$-$F$-inverse and that $\sigma$ is injective on $M(I_P^\lambda)$. Let $f\in\sigma^{-1}(g)$ be the maximal element. If we can show that $d(f)=d(\alpha_g)=(g^{-1})P\cap P$, then equation \eqref{eq:gradingaction} gives us that $f=\alpha_g$. We already noted that equation \eqref{eq:gradingaction} implies that $d(f)\subset d(\alpha_g)$, so let $p\in (g^{-1})P\cap P$. Then there is some $q\in P$ such that $p=g^{-1}q$, so $g=qp^{-1}$. Then $\sigma(\lambda_q\lambda_p^\ast)=g$ and $p\in pP=d(\lambda_q\lambda_p^\ast)$. Since $f$ is maximal in $\sigma^{-1}(g)$, $p\in d(\lambda_q\lambda_p^\ast)\subset d(f)$, so we are done.
\end{proof}

The computation of the $K$-theory for $C^\ast_r(P)$ in the case when $\mathcal{J}$ is independent and $P\subset G$ satisfies the Toeplitz condition is done in details in \cite{cuntz_echterhoff_li12}. The same computation could now be obtained using Theorem \ref{thm:maintheorem} and Proposition \ref{prop:toeplitz}.

\section{Graph inverse semigroups}
We will use \cite{raeburn05} as our main source on graphs and graph $C^\ast$-algebras. A (discrete) \emph{directed graph} $\mathcal{E}=(\mathcal{E}^0,\mathcal{E}^1,s,r)$ consists of countable sets $\mathcal{E}^0$, $\mathcal{E}^1$ and functions $s,r:\mathcal{E}^1\to\mathcal{E}^0$. The elements of $\mathcal{E}^0$ are called the \emph{vertices} of $\mathcal{E}$, while the elements of $\mathcal{E}^1$ are called the \emph{edges} of $\mathcal{E}$. A \emph{finite path} in $\mathcal{E}$ is either a single vertex or a finite string $\mu_n\ldots\mu_1$ of edges such that $r(\mu_k)=s(\mu_{k+1})$ for any $0\leq k\leq n-1$. Let $\mathcal{E}^\ast$ be the set of all finite paths in $\mathcal{E}$. For a path $\mu=\mu_n\cdots\mu_1\in\mathcal{E}^\ast$, define $s(\mu)=s(\mu_1)$ and $r(\mu)=r(\mu_n)$. If $\mu$ is a single vertex, then define $s(\mu)=r(\mu)=\mu$. One can define a partial product on $\mathcal{E}^\ast$ as follows. If $\mu$ and $\nu$ are strings of edges, then their product $\mu\nu$ is given by concatenation if the resulting string is a path, otherwise the product is undefined. If $\mu$ is a string of edges and $v$ is a vertex, then $\mu v=\mu$ if $s(\mu)=v$ and $v\mu=\mu$ if $r(\mu)=v$. If $v,w$ are vertices, then $vw=v$ if $v=w$.

In \cite{paterson02} (see also \cite{lawson99}) the \emph{graph inverse semigroup} $S_\mathcal{E}$ is defined to be the set
\[
S_{\mathcal{E}}=\{(\mu,\nu)\in \mathcal{E}^\ast\times \mathcal{E}^\ast:s(\mu)=s(\nu)\}\cup\{0\}
\]
with all products not involving $0$ given by
\[
(\mu,\nu)(\alpha,\beta)=\begin{cases}
(\mu,\beta\nu')\mbox{ if } \nu=\alpha\nu'\\
(\mu\alpha',\beta)\mbox{ if } \alpha=\nu\alpha'\\
0\mbox{ otherwise.}                         
\end{cases}
\]
The $\ast$-operation is given by $(\mu,\nu)^\ast=(\nu,\mu)$. For $(\mu,\nu)\in S_\mathcal{E}$, we have
\[
(\mu,\nu)^\ast (\mu,\nu)=(\nu,\mu)(\mu,\nu)=(\nu,\nu)
\]
This shows that $E(S_\mathcal{E})^\times=\{(\nu,\nu):\nu\in\mathcal{E}^\ast\}$. We have $(\mu,\nu)\leq (\alpha,\beta)$ if and only if there is some $\rho\in\mathcal{E}^\ast$ such that $\mu=\alpha\rho$ and $\nu=\beta\rho$. In particular we get that for $v\in\mathcal{E}^0$ we have $(\mu,\mu)\leq (v,v)$ if and only if $\mu=v\mu$ if and only if $r(\mu)=v$. The maximal elements of $S_\mathcal{E}^1$ are
\[
M(S_\mathcal{E}^1)=\{1\}\cup\{(\mu,\nu):\mu\mbox{ and }\nu\mbox{ have no common initial segment.}\}
\]
Every element is beneath such a maximal element, so $S_\mathcal{E}^1$ is $0$-$F$-inverse \cite{lawson01}. Let $\mathbb{F}$ be the free group on the alphabet $\mathcal{E}^1$. Define $h:\mathcal{E}^\ast\to\mathbb{F}$ by
\[
h(\mu)=\begin{cases}
1\mbox{ if }\mu\mbox{ is a vertex,}\\
\mu\mbox{ otherwise.}
\end{cases} 
\]
Define $\sigma:(S_\mathcal{E}^1)^\times\to\mathbb{F}$ by $\sigma((\mu,\nu))=h(\mu)h(\mu)^{-1}$. It is easy to check that $\sigma$ is a morphism and that it is injective on the maximal elements of $S_\mathcal{E}^1$. The next lemma lets us identify $\widetilde{E(S_\mathcal{E})}=E(S_\mathcal{E})/\approx$ with $\mathcal{E}^0$.

\begin{lemma}\label{lem:graphindex}
Let $\mu,\nu\in\mathcal{E}^\ast$. Then $(\mu,\mu)\approx (\nu,\nu)$ if and only if $s(\mu)=s(\nu)$.
\end{lemma}
\begin{proof}
Fix $\mu,\nu\in\mathcal{E}^\ast$. Suppose $s(\mu)=s(\nu)$. Then $(\nu,\nu)=(\mu,\nu)^\ast (\mu,\nu)$ and $(\mu,\mu)=(\mu,\nu)(\mu,\nu)^\ast$.

Conversely, suppose $(\nu,\nu)\approx (\mu,\mu)$. Then there exists $(\alpha,\beta)\in S_\mathcal{E}$ such that $(\nu,\nu)=(\alpha,\beta)^\ast (\alpha,\beta)=(\beta,\beta)$. and $(\mu,\mu)=(\alpha,\beta)(\alpha,\beta)^\ast=(\alpha,\alpha)$. It follows that $\nu=\beta$ and $\mu=\alpha$, so $s(\mu)=s(\alpha)=s(\beta)=s(\nu)$.
\end{proof}

\begin{lemma}\label{lem:graphstabilizer}
For each $\mu\in\mathcal{E}^\ast$, $G_{(\mu,\mu)}$ is the trivial group.
\end{lemma}
\begin{proof}
Suppose $(\alpha,\beta)\in S_\mathcal{E}$ satisfies $(\mu,\mu)=(\beta,\alpha)(\alpha,\beta)=(\alpha,\beta)(\beta,\alpha)$. Then clearly $\alpha=\beta=\mu$, so $\sigma((\alpha,\beta))=h(\mu)h(\mu)^{-1}=1$.
\end{proof}

It is possible to show that $C^\ast_r(S_\mathcal{E})$ is canonically isomorphic to the Toeplitz $C^\ast$-algebra of $\mathcal{E}$ as defined in \cite{raeburn05}. Thus the next proposition also follows from \cite[Theorem 1.1]{burgstaller09}.

\begin{proposition}
We have $K_0(C^\ast_r(S_\mathcal{E}))=\bigoplus_{\mathcal{E}^0}\mathbb{Z}$ and $K_1(C^\ast_r(S_\mathcal{E}))=0$.
\end{proposition}
\begin{proof}
By Lemma \ref{lem:graphstabilizer}, $G_e=\{1\}$ for each $e\in E(S_\mathcal{E})$. Then $K_0(C^\ast_r(G_e))=K_0(\mathbb{C})=\mathbb{Z}$ and $K_1(C^\ast_r(\{1\}))=K_1(\mathbb{C})=0$. Since $\mathbb{F}$ is free, it is a-T-menable \cite{aTmenability01}. The result now follows from Lemma \ref{lem:graphindex} and Theorem \ref{thm:maintheorem}.
\end{proof}

\section{Tiling inverse semigroups}
The tiling inverse semigroups were first introduced in \cite{kellendonk97,kellendonk97b}. A \emph{tile} in $\mathbb{R}^n$ is a subset of $\mathbb{R}^n$ homeomorphic to the closed unit ball. A \emph{partial tiling} is a collection of tiles that have pairwise disjoint interiors. The \emph{support} of a partial tiling is the union of its elements. A \emph{tiling} is a partial tiling with support equal to $\mathbb{R}^n$. A \emph{patch} is a finite partial tiling.

For any tile $t$ and $x\in\mathbb{R}^n$, let $t+x$ be the translation of $t$ by $x$. For any partial tiling $P$, let $P+x=\{t+x:t\in P\}$. Let $T$ be a tiling. Let $\mathcal{M}$ be the set of subpatches $P$ of $T$ such that $P$ has connected support. For $P,Q\in\mathcal{M}$ and $t_1,t_2\in P$, $r_1,r_2\in Q$, we say that $(t_1,P,t_2)\sim (r_1,Q,r_2)$ if there is some $x\in\mathbb{R}^n$ such that $t_1+x=r_1,t_2+x=r_2$ and $P+x=Q$. The equivalence class of $(t_1,P,t_2)$ is denoted $[t_1,P,t_2]$, and is called a \emph{doubly pointed pattern class}. Let
\[
S_T=\{[t_1,P,t_2]:P\in\mathcal{M},t_1,t_2\in P\}\cup\{0\}
\]
We define an inverse semigroup structure on $S_T$. Let $[t_1,P,t_2],[r_1,Q,r_2]\in S_T$. Whenever there are $x,y\in\mathbb{R}^n$ such that $P+x$ and $Q+y$ are patches in $T$ and $t_2+x=r_1+y$, define
\[
[t_1,P,t_2][r_1,Q,r_2]=[t_1+x,(P+x)\cup (Q+y),r_2+y].
\]
All other products are defined to be $0$. The $\ast$-operation is given by
\[
[t_1,P,t_2]^\ast=[t_2,P,t_1]
\]
$S_T$ is called the \emph{connected tiling semigroup} of $T$. If one drops the requirement that the elements of $\mathcal{M}$ have connected support, one gets the \emph{tiling semigroup} $\Gamma_T$.

It was shown in \cite{exel_goncalves_starling12} that when $T$ is so-called strongly aperiodic, repetitive and has finite local complexity (see the original paper for definitions), the tiling $C^\ast$-algebra $A_T$ from  \cite{kellendonk95,kellendonk_putnam00} is the tight $C^\ast$-algebra of $S_T$. We do not know if the reduced $C^\ast$-algebra of $S_T$ or $\Gamma_T$ has been previously studied.

Clearly, $E(S_T)=\{[t,P,t]:P\in\mathcal{M},t\in P\}$. It is easy to see that $[t,P,t]\leq [r,Q,r]$ if and only if there is an $x\in\mathbb{R}^n$ such that $r=t+x$ and $Q\subset P+x$. Say that two partial tilings are \emph{congruent} if one is a translation of the other. Let $L$ be the set of congruence classes of elements in $\mathcal{M}$. The next lemma lets us identify $\widetilde{E(S_T)}$ (respectively $\widetilde{E(\Gamma_T)}$) with $L$.

\begin{lemma}\label{lem:tiling1}
Let $P,Q\in\mathcal{M}$ and $t\in P$, $r\in Q$. Then $[t,P,t]\approx [r,Q,r]$ if and only if $P$ and $Q$ are congruent.
\end{lemma}
\begin{proof}
Suppose $P$ and $Q$ are congruent. Then there is some $x\in\mathbb{R}^n$ such that $Q=P+x$. So $t+x\in Q$. Moreover, it is easy to check that $[r,Q,r]=[t+x,Q,r]^\ast [t+x,Q,r]$ and
\[
[t+x,Q,r][t+x,Q,r]^\ast=[t+x,Q,t+x]=[t,Q-x,t]=[t,P,t].
\]
So $[t,P,t]\approx [r,Q,r]$.

Conversely, suppose $[t,P,t]\approx [r,Q,r]$. Then there is some $R\in\mathcal{M}$ and $a,b\in R$ such that
\begin{align*}
[t,P,t]=[a,R,b]^\ast [a,R,b]=[b,R,b],\\ 
[r,Q,r]=[a,R,b][a,R,b]^\ast=[a,R,a]
\end{align*}
It follows that $P$ and $Q$ are both congruent to $R$ and thus to each other.
\end{proof}

$S_T$ is not $0$-$F$-inverse in general. It is $0$-$F$-inverse when $n=1$ \cite[Proposition 4.2.1]{kellendonk_lawson04}. In this case, the universal grading of $S_T$ is well understood.

\begin{lemma}\label{lem:tiling2}
For each $e\in E(S_T)^\times$, $G_e$ is the trivial group.
\end{lemma}
\begin{proof}
Let $P\in\mathcal{M}$ and $t\in P$. Suppose that $R\in\mathcal{M}$ and $a,b\in R$ satisfy
\begin{align*}
[t,P,t]=[a,R,b]^\ast [a,R,b]=[b,R,b],\\ 
[t,P,t]=[a,R,b][a,R,b]^\ast=[a,R,a]
\end{align*}
This shows that $a=b$, so $[a,R,b]$ is idempotent, and $\sigma([a,R,b])=1$ regardless of what $\sigma$ is.
\end{proof} 

\begin{proposition}
Let $T$ be a one-dimensional tiling. Then $K_0(C^\ast_r(S_T))=\bigoplus_L\mathbb{Z}$ and $K_1(C^\ast_r(S_T))=0$.
\end{proposition}
\begin{proof}
By \cite[Corollary 4.2.4]{kellendonk_lawson04}, the universal group of $S_T$ is free, so it is a-T-menable.To show that $S_T$ is $0$-$F$-inverse it remains to check that the universal grading $\sigma$ is injective on the set of nonidempotent maximal elements in $S_T$. One has then to go through the actual construction of $\sigma$ in \cite{{kellendonk_lawson04}}. This requires some work, but is not hard. We leave this verification to the reader. The rest now follows from Lemmas \ref{lem:tiling1}, \ref{lem:tiling2} and Theorem \ref{thm:maintheorem}.
\end{proof}

Note also that for any $n$-dimensional tiling $T$, the tiling semigroup $\Gamma_T$ has the property that any element is beneath a maximal element \cite{kellendonk_lawson04}. If one can find a good group morphism for $\Gamma_T$, one will get similar results for the $K$-theory of $C^\ast_r(\Gamma_T)$. In \cite{kellendonk_lawson04} Kellendonk and Lawson also showcase other $0$-$F$-inverse semigroups, such as inverse semigroups of point sets in $\mathbb{R}^n$ and point set semigroups of model sets, where such morphisms exist.

\bibliographystyle{plain}
\bibliography{bibliography}

\end{document}